\documentclass[11pt,letterpaper]{amsart}
\usepackage{amsfonts,amssymb,amsmath,amsthm,mathtools,comment} 
\usepackage{url} 
\usepackage{enumerate} 
\usepackage{stmaryrd,color}
\usepackage{comment}
\usepackage{multicol}
\usepackage{mathrsfs}
\usepackage{mathtools}
\usepackage{comment}
\usepackage{color}

\newtheorem{thrm}{Theorem}[section]
\newtheorem{lem}[thrm]{Lemma}

\newtheorem{cor}[thrm]{Corollary}
\theoremstyle{definition}
\newtheorem{definition}[thrm]{Definition}
\newtheorem{remark}[thrm]{Remark}
\newtheorem{example}[thrm]{Example}
\numberwithin{equation}{section}

\newcommand{\bb}{\mathbb}
\newcommand{\mc}{\mathcal}

\DeclareMathOperator{\vspan}{span}
\DeclareMathOperator{\aff}{aff}
\DeclareMathOperator{\lin}{lin}
\DeclareMathOperator{\conv}{conv}
\DeclareMathOperator{\verti}{vert}

\title[Triangulating the Matroid Polytope]{A Regular Unimodular Triangulation of the Matroid Base Polytope}
\author{Spencer Backman}
\email{sbackman@uvm.edu}
\author{Gaku Liu}
\email{gakuliu@uw.edu}

\begin{document}

\begin{abstract}
    We produce the first regular unimodular triangulation of an arbitrary matroid base polytope.  We then extend our triangulation to integral generalized permutahedra. Prior to this work it was unknown whether each matroid base polytope admitted a unimodular cover.  
\end{abstract}

\maketitle

\section{Introduction}

Despite considerable interest, very little is known about triangulations of matroid base polytopes.  There are a few motivations for wanting to have nice triangulations of matroid base polytopes.  The first motivation comes from White's conjecture whose weakest version states that the toric ideal of a matroid base polytope is quadratically generated \cite{white1977basis}\cite{lason2014toric}.  Herzog and Hibi asked whether the toric ideal of every matroid base polytope has a quadratic Gr\"obner basis \cite{herzog2002discrete}.  It follows by a result of Sturmfels \cite{sturmfels1996grobner} combined with an observation of Ohsugi and Hibi \cite{ohsugi1999toric} that the existence of a quadratic Gr\"obner basis is equivalent to the existence of a quadratic triangulation, i.e. a regular unimodular flag triangulation.  The existence of a quadratic triangulation is known for base sortable matroids, e.g.  positroids \cite{stanley1977eulerian,sturmfels1996grobner,blum2001base, lam2007alcoved, lam2018alcoved}.

The second motivation comes from Ehrhart theory.  A formula for the volume of a matroid base polytope was calculated by Ardila--Doker--Benedetti \cite{ardila2010matroid}, but no formula is currently known which is cancellation free, i.e. involves no subtraction.  If a polytope  $P$ admits a unimodular triangulation $\mc T$, then the volume of $P$ is equal to the number of maximal simplices in $\mc T$.  The volume of a polytope occurs as the leading coefficient of the Ehrhart polynomial. Several researchers have investigated Ehrhart polynomials for matroid base polytopes \cite{castillo2018berline}\cite{jochemko2022generalized} \cite{ferroni2022ehrhart} largely motivated by the conjecture of De Loera--Haws--K\"oppe \cite{de2009ehrhart}  that matroid base polytopes are Ehrhart positive---this conjecture was recently disproven by Ferroni \cite{ferroni2022matroids}, but various other questions about these polynomials remain open.  The volume of a polytope $P$ is also given by the evaluation of the $h^*$-polynomial at 1.  Another conjecture by De Loera--Haws--K\"oppe, which remains open, is that the $h^*$-vectors of matroid base polytopes are unimodal \cite{de2009ehrhart}.  Ferroni further conjectures that the $h^*$-polynomial of a matroid polytope (more generally an integral generalized permutahedron) is real-rooted \cite{ferroni2022ehrhart,ferronicomm}.  It has been conjectured that if a polytope $P$ has the integer decomposition property (is IDP), then $P$ has a unimodal $h^\ast$-vector \cite{schepers2013unimodality}, and it is known that every matroid base polytope is IDP \cite{herzog2002discrete}.   We note that the property of admitting a unimodular triangulation is strictly stronger than the property of being IDP \cite{bruns1999normality}.  We refer the reader to \cite{ferroni2023examples} for a comprehensive survey of results in this area.  It is known that the $h^*$-vector of a polytope is equal to the $h$-vector of any unimodular triangulation of the polytope \cite{stanley1980decompositions}\cite{betke1985lattice}, thus one might hope that such a triangulation could shed some light on this conjecture.

A natural question which sits in between these various results and conjectures is whether each matroid base polytope admits a (not necessarily flag) regular unimodular triangulation.  That the matroid base polytope admits a (not necessarily regular) unimodular triangulation was conjectured by Haws in their 2009 thesis \cite{haws2009matroid}.  In this paper we give an affirmative answer to this question by providing a regular unimodular triangulation of an arbitrary matroid base polytope.  We then apply this result to produce a regular unimodular triangulation of an arbitrary integral generalized permutahedron, and explain how this gives a regular unimodular triangulation of the matroid independence polytope.  We emphasize that prior to this work it was unknown whether every matroid base polytope admitted a unimodular cover (this was also conjectured by Haws \cite{haws2009matroid}) let alone a unimodular triangulation.  Our construction produces many different triangulations, but at the time of writing we do not know if any of them are flag.  We invite other researchers to try their hand at applying our triangulation to the topics above.  See Remark \ref{whiteproblem}.

\section{Preliminaries}
We recommend the following texts for an introduction to matroids \cite{Oxley2011matroid}, polytope theory \cite{ziegler2012lectures}, and triangulations \cite{de2010triangulations}\cite{haase2021existence}.  Let $[n]$ denote the set of integers $\{1,\dots,n\}$.  Given $S \subseteq [n]$ we will employ the notation  $x_S := \sum_{i \in S} x_i$.  We identify $\{0,1\}^n$ with the collection of all subsets of $[n]$.  We denote the standard basis vectors for $\mathbb{R}^n$ by $e_i$ for $1\leq i\leq n$. 

\begin{definition}
    A \emph{matroid} is a pair $M =(E,\mathcal{B})$ where $E$ is a finite set called the \emph{ground set}, and $\mathcal{B}$ is a nonempty collection of subsets of $E$ called the \emph{bases} which satisfy the following basis exchange condition:

    \begin{itemize}
        \item For any $B_1,B_2 \in \mathcal{B}$ and $x \in B_1\setminus B_2$, there exists some $y \in B_2 \setminus B_1$ such that $(B_1 \setminus \{x\}) \cup \{y\} \in \mathcal{B}$.
    \end{itemize}

A set $I\subseteq E$ is \emph{independent} if there exists some basis $B \in \mc B$ such that $I \subseteq B$.  The collection of independent sets is denoted $\mc I$.  The \emph{rank} of a set $S \subseteq E$, written $r(S)$, is the maximum cardinality of an independent set contained in $S$.
\end{definition}

Matroid independence  polytopes and the matroid base polytopes were introduced by Edmonds \cite{edmonds2003submodular}.

\begin{definition}
    Given a matroid $M$ on ground set $E = [n]$, the \emph{matroid base polytope} $P_M$ is the convex hull of the indicator vectors for the bases of $M$, and the \emph{matroid independence polytope} $P_{\mc I}$ is the convex hull of the indicator vectors of the independent sets.  More explicitly, given $S\subseteq E$, we define the indicator vector $\chi_S \in \mathbb{R}^n$ by
\[ \chi_S(i)=
  \begin{cases} 
      1 & i \in S \\
      0 & i \notin S
      \end{cases} \]
Thus $P_M = \conv\{\chi_B: B \in \mc B\}$ and $P_{\mc I}= \conv\{\chi_I: I \in \mc I\}$.

\end{definition}

The matroid base polytope is the distinguished face of the matroid independence polytope where the sum of the coordinates is maximized.  The matroid independence polytope will be discussed at the end of this article (see Corollary \ref{indep}).

Gelfand--Goresky--MacPherson--Serganova uncovered a connection between matroid base polytopes and the geometry of the Grassmannian \cite{gelfand1987combinatorial}.   They showed that torus orbit closure of a linear space $L$  in the Grassmannian is a normal toric variety whose weight polytope is the matroid base polytope $P_{M(L)}$, where $M(L)$ is the matroid determined by $L$.  See Katz \cite{katz2016matroid} for an overview of this story.  By standard toric theory, our regular unimodular triangulation of $P_M$ gives a projective Crepant resolution of the toric variety associated to the cone over a matroid base polytope.

Matroid bases polytopes allow for a polytopal characterization of matroids.

\begin{thrm}\label{gelmat} \cite{edmonds2003submodular}\cite{gelfand1987combinatorial}
A polytope $P$ is a matroid base polytope for some matroid $M$ if and only if $P$ is a 0-1 polytope whose edge directions are of the form $e_i-e_j$.
\end{thrm}

Polymatroids are a generalization of matroids described by monotonic submodular fuctions taking values in the nonnegative reals.  Their base polytopes are equivalent by translation to the generalized permutahedra of Postnikov \cite{postnikov2009permutohedra}.  See \cite{ardila2020coxeter} for a careful treatment of the following definition. 

\begin{definition}\label{genper}
    A \emph{generalized permutahedron} $P \subseteq \mathbb{R}^n$ is a polytope defined by any one of the following equivalent conditions:

    \begin{enumerate}
        \item \label{cond1} The edge directions for $P$ are all of the form $e_i-e_j$,
        \item\label{cond2} The normal fan of $P$ is a coarsening of the braid arrangement,
        \item\label{cond3}$P$ is defined by inequalities $x_S \leq f(S)$ where $f:\{0,1\}^n\rightarrow \mathbb{R}$ is a submodular function, together with a single equation $x_{[n]} =f([n])$.
    \end{enumerate}
\end{definition}

An \emph{integral generalized permutahedron} $P$ is a generalized permutahedron whose vertices have integer coordinates.  The following is well-known, and follows from the unimodularity of the set of primitive ray generators of each chamber in the braid arrangement.

\begin{lem}\label{permint}
Let $P$ be a generalized permutahedron determined by a submodular function $f$ as in condition (\ref{cond3}) of Definition \ref{genper}.  If $f$ is an integer-valued function then $P$ is an integral generalized permutahedron.  Moreover, if $P$ is an integral generalized permutahedron then $f$ may be chosen to be integer-valued.
\end{lem}

In our proof, we will use condition (\ref{cond2}) from Definition \ref{genper} as this allows us to describe the affine span of a face of a matroid base polytope.

\begin{lem}\label{flag}
Let $P$ be an integral generalized permutahedron and $\aff(P)$ its affine span. Then

    \[\aff(P) = \bigcap_{i=1}^j \{ x_{S_i} = b_i \}
\]
for some flag of subsets $\emptyset = S_0 \subsetneq S_1 \subsetneq \dots \subsetneq S_j = [n]$ and some $b_i \in \bb Z$.
\end{lem}
  
We note that when $P$ is a matroid base polytope, the $b_i$ in the lemma above is equal to the rank of the set $S_i$ viewed as a subset of the ground set of the matroid. 

\begin{definition}

    A \emph{subdivision} of a polytope $P$ is a collection of polytopes $\mathcal{S} = \{P_1, \ldots, P_k\}$ such that
    
    \begin{enumerate}
        \item $\bigcup_{i=1}^k P_i =P$ 
        \item for each $P_i \in \mathcal{S}$ and $F$ a face of $P_i$, there exists some $j$ such that $F=P_j$ 
        \item for any $i$ and $j$ with $1\leq i,j \leq k$, the intersection $P_i\cap P_j$ is a face of both $P_i$ and $P_j.$
        
    \end{enumerate}

A maximal polytope in $\mathcal{S}$ is a \emph{cell} of $\mathcal{S}$.
\end{definition}
\begin{definition}
    A \emph{triangulation} of a polytope $P$ is a subdivision $\mathcal{T} = \{T_1, \ldots, T_k\}$ of $P$ such that each polytope $T_i$ is a simplex.
\end{definition}

\begin{definition}
    Let $P \subset \bb R^n$ be a polytope and $S$ a finite subset of $P$ containing the vertices of $P$. Given a function $f : S \to \bb R$, the subdivision \emph{induced} by $f$ is the subdivision of $P$ formed by projecting the lower faces of the polytope
    \[
    \conv\{ (x,f(x)) : x \in S \} \subset \bb R^{n+1}.
    \]    
    A subdivision is \emph{regular} if it is induced by some function $f$.
\end{definition}

Given a set $S \subseteq \bb R^n$, let $\aff(S)$ denote the affine span of $S$. Let $\lin(S)$ denote the linear subspace of $\bb R^n$ with the same dimension and parallel to $\aff(S)$.


\begin{definition}
    A lattice simplex $T$ is \emph{unimodular} if it has normalized volume 1.  Equivalently, if $T$ has vertices $v_0,\ldots, v_n \in \mathbb{Z}^n$, then $T$ is unimodular whenever a maximal linearly independent set of edge vectors $\{v_i-v_j\}$ form a lattice basis for  $\lin(T) \cap \mathbb{Z}^n$.  
\end{definition}

\begin{definition}
    The \emph{resonance arrangement} $\mathcal{A}_n$ is the hyperplane arrangement in $\mathbb{R}^n$ consisting of all hyperplanes $H_S =\{x \in \mathbb{R}^n: x_S = 0\}$ where $\emptyset \subsetneq S\subseteq [n]$.
\end{definition}

For an introduction to the resonance arrangement (also called the \emph{all subsets arrangement}) we refer the reader to \cite{kuhne2023universality}.  A \emph{flat} of a hyperplane arrangement $\mathcal{H}$ is an intersection of hyperplanes in $\mathcal{H}$.

\begin{definition}\label{generic}
    We say that an affine functional $\ell:\mathbb{R}^n\rightarrow \mathbb{R}$ is \emph{generic} if it is non constant on each positive dimensional flat of the resonance arrangement.

\end{definition}

We note that a generic point $p$ on the $n$-th moment curve
\[
C_n = \{(t, t^2, \ldots, t^{n}): t \in \mathbb{R}\}
\]
produces a generic linear functional $x \mapsto \langle x,p \rangle$.

\section{A deletion-contraction triangulation}

In this section we establish the main result of this paper. 

\begin{thrm}\label{main}
Every matroid base polytope has a regular unimodular triangulation.
\end{thrm}

Before providing a proof, we briefly give some context for our construction.  Two fundamental operations on a matroid are the deletion and contraction of an element, and many important constructions in matroid theory proceed by an inductive appeal to these operations.  If $e$ is a loop or coloop, then the matroid base polytope $P_M$ is translation equivalent to $P_{M/e}$ and $P_{M\setminus e}$.   If $e$ is neither a loop nor a coloop then $P_M$ is the convex hull of $P_{M/e}$ and $P_{M\setminus e}$.  In this way, our recursive construction fits into the paradigm of deletion-contraction. 

\begin{proof}[Proof of Theorem~\ref{main}]
Let $M=(E,\mathcal{B})$ be a matroid with ground set $E=[n]$, and $P_M \subset \bb R^n$ its matroid base polytope. We will use $\verti(P_M)$ to denote the vertices of $P_M$. We show $P_M$ has a unimodular triangulation by induction on $n$. If $n = 1$, then $P_M$ is a point and we are done.

Assume $n \ge 2$.  Let $P_0$ and $P_1$ be the polytopes in $\bb R^{n-1}$ such that $P_0 \times \{0\} = P_M \cap \{x_1 = 0\}$ and $P_1 \times \{1\} = P_M \cap \{x_1 = 1\}$. Note that $P_0$ or $P_1$ may be empty, which occurs if 1 is a loop or coloop. If $P_0$ is nonempty then it is the matroid base polytope of $M \setminus 1$, and if $P_1$ is nonempty then it is the matroid base polytope of $M / 1$.

By the inductive hypothesis, $P_0$ and $P_1$ have regular unimodular triangulations. (We assume an empty polytope has a regular unimodular triangulation induced by a function with empty domain.) Let $f_0 : \verti(P_0) \to \bb R$ and $f_1 : \verti(P_1) \to \bb R$ be functions which induce these triangulations. Let $\ell_0, \ell_1 : \bb R^{n-1} \to \bb R$ be affine functionals such that $\ell_0 - \ell_1$ is generic.  Let $\epsilon>0$ be sufficiently small, and define $f : \verti(P_M) \to \bb R$ to be the function
\[
f(x) = \begin{dcases*}
\ell_0(x_2,\dots,x_n) + \epsilon f_0(x_2,\dots,x_n) & if $x_1 = 0$ \\
\ell_1(x_2,\dots,x_n) + \epsilon f_1(x_2,\dots,x_n) & if $x_1 = 1$.
\end{dcases*}
\]

We claim that $f$ induces a unimodular triangulation of $P_M$\footnote{That the subdivision induced by $f$ is independent of $\epsilon$ for $\epsilon$ sufficiently small is guaranteed by the existence of the secondary fan of $P_M$.}. We will need the following lemmas. 

We say that two linear subspaces $V$, $W$ of $\bb R^n$ are \emph{independent} if $V \cap W = \emptyset$.
We say that two rational linear subspaces $V$, $W$ of $\bb R^n$ are \emph{complementary} if $(V \cap \bb Z^n) + (W \cap \bb Z^n) = (V+W) \cap \bb Z^n$. Equivalently, $V$ and $W$ are complementary if there exist $A \subset (V \cap \bb Z^n)$ and $B \subset (W \cap \bb Z^n)$ such that $A \cup B$ generates $(V+W) \cap \bb Z^n$ over the integers. 

\begin{lem}\label{rational}
    Let $V$ and $W$ be independent complementary rational linear subspaces in $\mathbb{R}^n$.  If $X$ and $Y$ are rational subspaces of $V$ and $W$, respectively, then $X$ and $Y$ are complementary. 
\end{lem}

\begin{proof}
    Take lattice bases $B_X$ of $X \cap \bb Z^n$ and $B_Y$ of $Y \cap \bb Z^n$, and extend them to lattice bases $B_V$ of $V \cap \bb Z^n$ and $B_W$ of $W \cap \bb Z^n$, respectively.  Because $V$ and $W$ are independent and complementary, $B_V \cup B_W$ gives a basis for $(V + W) \cap \bb Z^n$.  For any $x \in (X+Y) \cap \bb Z^n$ we know that $x$ is an integral combination of elements of $B_V \cup B_W$ and thus must be an integral combination of elements of $B_X \cup B_Y$.  Hence $X$ and $Y$ are complementary.
\end{proof}


We warn the reader that $P$ and $Q$ in the following lemma will not correspond to $P_0$ and $P_1$.

\begin{lem} \label{complementary}
Let $P$, $Q$ be two matroid base polytopes. Then $\lin(P)$ and $\lin(Q)$ are complementary.
\end{lem}

\begin{proof}

It is well-known that $\{e_i-e_j\}_{1 \le i < j \le n}$ is a \emph{totally unimodular system}, i.e., every subset $U \subset \{e_i-e_j\}_{1 \le i < j \le n}$ is a generating set for the lattice $(\vspan_{\bb R}{U}) \cap \bb Z^n$.

Let $A \subset \lin(P)$ be the set of edge directions of $P$ and let $B \subset \lin(Q)$ be the set of edge directions of $Q$. Since all edge directions of $P$ and $Q$ are of the form $e_i-e_j$ and $\{e_i-e_j\}_{1 \le i < j \le n}$ is a totally unimodular system, we have that $A \cup B$ generates the lattice $\vspan_{\bb R}(A \cup B) \cap \bb Z^n = (\lin(P) + \lin(Q)) \cap \bb Z^n$, as desired.
\end{proof}

We now return to the main proof. If either $P_0$ or $P_1$ is empty, then $f$ induces a unimodular triangulation on $P_M$ by the definition of $f_0$ and $f_1$. Thus, assume $P_0$ and $P_1$ are not empty. Let $g : \verti(P) \to \bb R$ be the function
\[
g(x) = \begin{dcases*}
\ell_0(x_2,\dots,x_n) & if $x_1 = 0$ \\
\ell_1(x_2,\dots,x_n) & if $x_1 = 1$.
\end{dcases*}
\]
Let $\mc S$ be the subdivision of $P_M$ induced by $g$. Since $g$ is affine on $P_0$ and $P_1$, it must be that $P_0, P_1 \in \mc S$. Hence every cell of $\mc S$ is of the form $\conv( F_0 \times \{0\} \cup F_1 \times \{1\} )$, where $F_0$ and $F_1$ are faces of $P_0$ and $P_1$, respectively. 

We claim that because $\ell_0 - \ell_1$ is generic, $\lin(F_0)$ and $\lin(F_1)$ are independent. Note that the function $g$ and the function
\[
g'(x) = \begin{dcases*}
0 & if $x_1 = 0$ \\
\ell_1(x_2,\dots,x_n) - \ell_0(x_2,\dots,x_n) & if $x_1 = 1$.
\end{dcases*}
\]
differ by an affine function on $\bb R^n$, and therefore they induce the same subdivision on $P$. Moreover, $0 - (\ell_1 - \ell_0) = \ell_0 - \ell_1$. Therefore, by replacing $g$ with $g'$, we may assume $\ell_0 = 0$.
 
 Let $A_0 = \aff(F_0)$ and $A_1 = \aff(F_1)$ and suppose that $\lin(A_0)$ and $\lin(A_1)$ are not independent.  Then $L = \lin(A_0) \cap \lin(A_1)$ is a positive dimensional linear space, and by Lemma \ref{flag} it is a flat in the resonance arrangement. Furthermore, $x_0+L \subset A_0$ and $x_1+L \subset A_1$ for some $x_0$, $x_1 \in \bb Z^{n-1}$. Now, let $\tilde{g} : \bb R^n \to \bb R$ be an affine function which agrees with $g$ on $\conv( F_0 \times \{0\} \cup F_1 \times \{1\} )$. (This exists by the assumption that $\conv( F_0 \times \{0\} \cup F_1 \times \{1\} )$ is a cell in the subdivision induced by $g$.) Since $\ell_0 = 0$, we have that $\tilde{g}$ is 0 on $A_0 \times \{0\}$, and hence it is 0 on $(x_0+L)\times\{0\} \subset A_0 \times \{0\}$. Since $\tilde{g}$ is affine, it is therefore constant on $(x_1+L) \times \{1\}$. Since $(x_1+L) \times \{1\} \subset A_1 \times \{1\}$, it follows that $\ell_1$ is constant on $x_1+L$, and hence it is constant on $L$. This contradicts the assumption that $\ell_0-\ell_1$ is non-constant on positive dimensional flats of the resonance arrangement.  Thus $\lin(F_0)$ and $\lin(F_1)$ are independent. Each face of a matroid base polytope is a matroid base polytope, hence by Lemma \ref{complementary}, $\lin(F_0)$ and $\lin(F_1)$ are complementary.

Now, let $\mc T$ be the subdivision of $P_M$ induced by $f$.  Because $f$ is a perturbation of $g$, it follows that $\mc T$ refines $\mc S$.  Furthermore, since $f$ restricted to $P_0 \times \{0\}$ is an affine function plus $\epsilon f_1(x_2,\dots,x_n)$, the restriction of $\mc T$ to $P_0 \times \{0\}$ is a unimodular triangulation. Similarly the restriction of $\mc T$ to $P_1 \times \{1\}$ is a unimodualar triangulation. Let $T$ be a cell of $\mc T$, and let $S = \conv( F_0 \times \{0\} \cup F_1 \times \{1\} )$ be a cell of $\mc S$ containing $T$. We have $T = \conv( T_0 \times \{0\} \cup T_1 \times \{1\} )$, where $T_i$ is a unimodular simplex contained in $F_i$.  Because $\lin(F_0)$ and $\lin(F_1)$ are independent, $\lin(T_0)$ and $\lin(T_1)$ are independent, hence $T$ is a simplex.  Moreover, since $\lin(F_0)$ and $\lin(F_1)$ are complementary and independent, it follows by Lemma \ref{rational} that $\lin(T_0)$ and $\lin(T_1)$ are complementary.  Thus we can form a basis $B$ of $(\lin(T_0)+\lin(T_1)) \cap \bb Z^{n-1}$ from maximal collections of linearly independent edge vectors of $T_0$ and $T_1$.  Each edge vector of $T$ from a vertex in $T_0$ to a vertex in $T_1$ has first coordinate 1. We claim that we can add any such edge vector $v$ to $B$ to give a basis $B'$ of $\lin(T) \cap \bb Z^n$.  Given $x \in \lin(T) \cap \bb Z^n$ we explain that $x$ is in the integral span of $B'$.  We can subtract an integer multiple of $v$ from $x$ so that the first coordinate is zero.  The resulting vector must be in $\lin(T_0)+\lin(T_1)$ and is integral, hence it is in the integral span of $B$.  We conclude that $T$ is a unimodular simplex and $\mc T$ is a unimodular triangulation.
\end{proof}


\begin{remark}
We would like to point out the connection of our proof to Edmond's matroid intersection theorem \cite{edmonds2003submodular} and Lemma 4.15 of Haase et al. \cite{haase2021existence}. The matroid intersection theorem states that the intersection of two matroid polytopes is an integral polytope. For the purposes of this discussion, an \emph{alcoved polytope} is an integral polytope in $\{x \in \bb R^n : \sum x_i = 0\}$ whose facet directions belong to $\{e_i-e_j\}_{1 \le i,j \le n}$.\footnote{The usual definition of an alcoved polytope is unimodularly equivalent to this one.}
Lemma 4.15 of \cite{haase2021existence} implies that if $P_0$, $P_1$ are alcoved polytopes, then $\conv(P_0 \times \{0\} \cup P_1 \times \{1\})$ has a unimodular triangulation.

Let $\mathcal{A}$ be a family of integral polytopes in $\bb R^n$ closed under taking faces and translation by integral vectors. For a positive integer $k$, consider the following two statements:
\begin{enumerate}[(A)]
\item For any $P_1$, \dots, $P_k \in \mathcal{A}$ which have unimodular triangulations, $\conv(P_1 \times e_1, \dots, P_k \times e_k)$ has a unimodular triangulation, where $e_1$, \dots, $e_k$ are the vertices of a unimodular simplex in $\bb R^{k-1}$.
\item For any $P_1$, \dots, $P_k \in \mathcal{A}$, $P_1 \cap \dots \cap P_k$ is an integral polytope.
\end{enumerate}
For example, when $k = 2$ and $\mc A$ is the set of integral translations of matroid polytopes, (A) is implied by the proof of Theorem~\ref{main} and (B) is the matroid intersection theorem. When $k =2$ and $\mc A$ is the set of alcoved polytopes, (A) is implied by Lemma 4.15 of \cite{haase2021existence} and (B) follows because the intersection of alcoved polytopes is alcoved. (This uses the total unimodularity of $\{e_i-e_j\}_{1 \le i,j \le n}$.)

We say that rational subspaces $W_1$, \dots, $W_k \subset \bb R^n$ are \emph{complementary}  if $\sum (W_i \cap \bb Z^n) = (\sum W_i) \cap \bb Z^n$.
Then (A) and (B) are equivalent to the following two statements, respectively. 
\begin{enumerate}[(A$'$)]
\item For any $P_1$, \dots, $P_k \in \mathcal{A}$, $\lin(P_1)$, \dots, $\lin(P_k)$ are complementary. 
\item For any $P_1$, \dots, $P_k \in \mathcal{A}$, $\lin(P_1)^\perp$, \dots, $\lin(P_k)^\perp$ are complementary.
\end{enumerate}
The equivalence (A) $\iff$ (A$'$) can be proven by modifying our proof of Theorem~\ref{main}. The equivalence (B) $\iff$ (B$'$) is standard from integer programming.

If $k = 2$, then (A$'$) and (B$'$) are equivalent, since rational subspaces $V$, $W$ are complementary if and only if $V^\perp$, $W^\perp$ are. Thus for $k = 2$ (A) and (B) are equivalent for any such family $\mc A$.

If $P$ is a matroid polytope, then $\lin(P)$ is spanned by vectors of the form $e_i-e_j$, while if $P$ is alcoved then $\lin(P)^\perp$ is spanned by vectors of the form $e_i-e_j$. Thus, (A$'$) for matroid polytopes and (B$'$) for alcoved polytopes are true for all $k$ by the total unimodularity of $\{e_i-e_j\}$. Thus, for all $k$, (A) is true for matroid polytopes and (B) is true for alcoved polytopes. By the previous paragraph, it follows that (A) and (B) are both true for $k =2$ for both of these families. (This gives an alternate proof of Lemma 4.15 of \cite{haase2021existence}.)  

On the other hand, for $k > 2$, (B$'$) is not true for matroid polytopes and (A$'$) is not true for alcoved polytopes. This is reflected in the fact that the matroid intersection theorem fails for intersections of three polytopes, while Lemma 4.15 of \cite{haase2021existence} fails for three alcoved polytopes, as was noted by those authors.
\end{remark}

The proof of Theorem \ref{main} implies the following more explicit statement: 

\begin{thrm} \label{functiondesc}
Let $P \in \bb R^n$ be a matroid base polytope. For each string $s \in \bigsqcup_{k=1}^{n-1} \{0,1\}^k$, let $\ell_s : \bb R^{n-|s|} \to \bb R$ be an affine functional, where $|s|$ is the length of $s$. Assume that $\ell_{s'0} - \ell_{s'1}$ is generic for all strings $s'$. Then for $1 \gg \epsilon_1 \gg \epsilon_2 \gg \dots \gg \epsilon_{n-1} > 0$, the function $f : \verti(P) \to \bb R$ defined by
\[
f(x) = \sum_{k=1}^{n-1} \epsilon_k \ell_{x_1 \dots x_k}(x_{k+1},\dots,x_n)
\]
induces a regular unimodular triangulation on $P_M$.
\end{thrm}

\begin{proof}
This is obtained by unwinding the induction in the proof of Theorem~\ref{main}. 
\end{proof}

We now explain how to extend our triangulation to all integral generalized permutahedra.

\begin{cor}
Every integral generalized permutahedron has a regular unimodular triangulation.
\end{cor}

\begin{proof}
Let $P \in \bb R^n$ be an integral generalized permutahedron. By translating $P$ if necessary, we may assume without loss of generality that there is some positive integer $R$ such that $P \subset \{x : 0 \le x_k \le R \text{ for all } 1 \le k \le n \}$. It is known that dicing $P$ by the hyperplanes $\{x_k = c\}$ where $c$ and $k$ are integers with $1 \le k \le n$ and $0 \le c \le R$ gives a regular integral subdivision $\mc X$ of $P$, and every cell of the subdivision is a translation of a matroid base polytope\footnote{\,This can be verified by appealing to the submodularity description of generalized permutahedra, Lemma \ref{permint}, and Theorem \ref{gelmat}.}. Let $g : P \cap \bb Z^n \to \bb R$ be a function which induces $\mc X$.

For each $s \in \bigsqcup_{k=1}^{n-1} \{0,\dots,R\}^k$, choose an affine functional $\ell_s : \bb R^{n-|s|} \to \bb R$ so that $\ell_{s'i} - \ell_{s'(i+1)}$ is generic for all strings $s'$ and integers $i$. For $1 \gg \epsilon_1 \gg \epsilon_2 \gg \dots \gg \epsilon_{n-1} > 0$, define the function $f : P \cap \bb Z^n \to \bb R$ by
\[
f(x) = g(x) + \sum_{k=1}^{n-1} \epsilon_k \ell_{x_1 \dots x_k}(x_{k+1},\dots,x_n).
\]

Then $f$ induces a subdivision of $P$ which refines $\mc X$. Moreover, by Theorem~\ref{functiondesc}, the restriction of $f$ to each cell of $\mc X$ induces a unimodular triangulation.
\end{proof}

\begin{cor}\label{indep}
    Every matroid independence polytope has a regular unimodular triangulation.
\end{cor}

\begin{proof}
Each matroid independence polytope $P_{\mathcal{I}}$ is unimodularily equivalent to an integral generalized permutahedron:  given a point $v =(v_1, \ldots v_n) \in P_{\mathcal{I}}$, let $\psi(v) = (v_0, v_1, \dots v_n) \in \mathbb{R}^{n+1}$, where $v_0 = r(E)-\sum_{i=1}^n v_i$.  The map $\psi$ is unimodular and its image is an integral generalized permutahedron\footnote{\,It is implicit in \cite{ardila2013lifted} that the independence polytope is unimodularily equivalent to a generalized permutahedron.}.  We can  apply our triangulation to $\psi(P_{\mathcal{I}})$ and then map this triangulation back to $P_{\mathcal{I}}$ to obtain a regular unimodular triangulation of the latter.
\end{proof}

\begin{example}\label{mainexample}
    
We provide an example of our triangulation for the cycle matroid of the complete graph $K_4$. Let $\verti(K_4) = \{v_0,v_1,v_2,v_3\}$.  To simplify notation we denote the edges of $K_4$ by integers:  $$v_0v_1=0,\, v_1v_2=1,\, v_0v_1=2,\, v_1v_3=3,\, v_0v_3=4,\, v_2v_3=5.$$ The bases are in the following order:
\begin{multicols}{3}
\begin{enumerate}\addtocounter{enumi}{-1}
\item \{0 1 3\}
\item \{1 2 3\}
\item \{1 3 4\}
\item \{0 1 4\}
\item \{0 1 5\}
\item \{1 2 5\}
\item \{1 4 5\}
\item \{1 2 4\}
\item \{0 2 4\}
\item \{2 3 4\}
\item \{0 2 3\}
\item \{0 2 5\}
\item \{2 3 5\}
\item \{0 3 5\}
\item \{3 4 5\}
\item \{0 4 5\}
\end{enumerate}
\end{multicols}

We take the height function described in Theorem \ref{functiondesc} as follows:  if $s$ is a string ending is 0, the function $\ell_s$ is 0. If a string ends in 1, and the string has length $k$, the function $\ell_s= (-3^{n-k-1}, -3^{n-k-2},\ldots,1)$. The cells of the associated triangulation are

\begin{multicols}{3}
\noindent \{3 7 8 9 12 14\}\\
\{3 5 7 8 12 14\}\\
\{3 5 6 7 8 14\}\\
\{3 5 8 11 12 14\}\\
\{3 5 6 8 11 14\}\\
\{3 4 6 11 14 15\}\\
\{3 4 5 6 11 14\}\\
\{3 4 5 11 12 14\}\\
\{3 6 8 11 14 15\}\\
\{0 3 8 11 14 15\}\\
\{0 3 4 5 11 12\}\\
\{0 3 4 5 12 14\}\\
\{0 3 4 11 12 14\}\\
\{0 3 4 5 6 14\}\\
\{0 3 4 11 14 15\}\\
\{0 4 11 12 13 14\}\\
\{0 4 11 13 14 15\}\\
\{0 3 5 8 11 12\}\\
\{0 3 8 11 12 14\}\\
\{0 3 5 6 7 14\}\\
\{0 10 11 12 13 14\}\\
\{0 8 10 11 14 15\}\\
\{0 8 10 11 12 14\}\\
\{0 8 9 10 12 14\}\\
\{0 10 11 13 14 15\}\\
\{0 3 5 7 12 14\}\\
\{0 3 5 7 8 12\}\\
\{0 2 3 6 7 14\}\\
\{0 2 3 7 9 14\}\\
\{0 1 7 9 10 12\}\\
\{0 1 5 7 10 12\}\\
\{0 5 8 10 11 12\}\\
\{0 7 8 9 10 12\}\\
\{0 5 7 8 10 12\}\\
\{0 1 2 6 7 14\}\\
\{0 1 2 7 9 14\}\\
\{0 1 7 9 12 14\}\\
\{0 1 5 7 12 14\}\\
\{0 1 5 6 7 14\}\\
\{0 3 7 9 12 14\}\\
\{0 3 8 9 12 14\}\\
\{0 3 7 8 9 12\}.
\end{multicols}

\end{example}

\begin{remark}\label{whiteproblem}
The authors, Matt Larson, and Sam Payne attempted to apply the construction of this article to produce quadratic triangulations of graphic matroid base polytopes, i.e. spanning tree polytopes.  We convinced ourselves that it not possible to do so using only $\ell_s$ above which are exponential.  We welcome others to attempt to apply our triangulation to White's conjecture.
\end{remark}

\section{Acknowledgements}
The authors would like to thank the University of Vermont, the University of Washington, and the Simons Center for Geometry and Physics Workshop ``Combinatorics and Geometry of Convex Polyhedra" for excellent working conditions where part of this work was completed.  The authors thank Mateusz Michalek for noting some typos in an earlier draft of this work, Luis Ferroni for providing further references, Matt Larson for providing Example \ref{mainexample}, and Sam Payne for suggesting a simplification of the proof of Lemma~\ref{complementary}.  The first author was supported by a Simons Collaboration Gift \# 854037 and an NSF Grant (DMS-2246967).

\bibliographystyle{alpha}
\small
\bibliography{references}

\end{document}